\documentclass[10pt,a4paper]{amsart}

\usepackage{amsmath}
\usepackage{amsthm}
\usepackage{enumerate}
\usepackage{fancyhdr}
\usepackage{amssymb}
\usepackage{graphicx}
\usepackage[all]{xy}
\usepackage{hyperref}
\usepackage{aliascnt}
\usepackage{stmaryrd}
\usepackage{braket}
\usepackage{xcolor}
\usepackage{mdframed}
\usepackage{mathtools}
\usepackage{verbatim}
\usepackage{txfonts}
\usepackage{setspace}   %Allows double spacing with the \doublespacing command

\newtheorem{lma}{Lemma}[section]

\newaliascnt{thmCt}{lma}
\newtheorem{thm}[thmCt]{Theorem}
\aliascntresetthe{thmCt}

\newaliascnt{corCt}{lma}
\newtheorem{cor}[corCt]{Corollary}
\aliascntresetthe{corCt}

\newaliascnt{propCt}{lma}
\newtheorem{prop}[propCt]{Proposition}
\aliascntresetthe{propCt}

\newtheorem*{thm*}{Theorem}
\newtheorem*{cor*}{Corollary}
\newtheorem*{prop*}{Proposition}

\theoremstyle{definition}

\newaliascnt{prgCt}{lma}
\newtheorem{prg}[prgCt]{}
\aliascntresetthe{prgCt}

\newaliascnt{dfnCt}{lma}
\newtheorem{dfn}[dfnCt]{Definition}
\aliascntresetthe{dfnCt}

\newaliascnt{rmkCt}{lma}
\newtheorem{rmk}[rmkCt]{Remark}
\aliascntresetthe{rmkCt}

\newaliascnt{rmksCt}{lma}

\aliascntresetthe{rmksCt}

\newaliascnt{ntnCt}{lma}

\aliascntresetthe{ntnCt}

\newaliascnt{qstCt}{lma}

\aliascntresetthe{qstCt}

\newaliascnt{prblCt}{lma}

\aliascntresetthe{prblCt}

\newaliascnt{exaCt}{lma}

\aliascntresetthe{exaCt}

\newcommand{\N}{\mathbb{N}}

\newcommand{\Z}{\mathbb{Z}}
\newcommand{\K}{\mathrm{K}}

\DeclareMathOperator{\im}{im}

\DeclareMathOperator{\her}{her}

\newcommand{\CatCa}{C^*}
\newcommand{\CatPoM}{\mathrm{PoM}}
\newcommand{\CatoM}{\mathrm{Mon}_\leq}

\DeclareMathOperator{\sr}{sr}

\DeclareMathOperator{\Lat}{Lat}
\DeclareMathOperator{\Idl}{Idl}
\DeclareMathOperator{\AbGp}{AbGp}
\DeclareMathOperator{\Cu}{Cu}

\newcommand{\hooklongrightarrow}{\lhook\joinrel\longrightarrow}

\begin{document}
\onehalfspacing
\title{Unitary Cuntz semigroups of ideals and quotients}

\author{Laurent Cantier}

\address{Laurent Cantier,
Departament de Matem\`{a}tiques \\
Universitat Aut\`{o}noma de Barcelona \\
08193 Bellaterra, Barcelona, Spain
}
\email[]{lcantier@mat.uab.cat}

\thanks{The author was supported by MINECO through the grant BES-2016-077192 and partially supported by the grants MDM-2014-0445 and MTM-2017-83487 at the Centre de Recerca Matem\`atica in Barcelona.}

\begin{abstract}
We define a notion of ideal for objects in the category of abstract unitary Cuntz semigroups  introduced in \cite{C20one} and termed $\Cu^\sim$. We show that the set of ideals of a $\Cu^\sim$-semigroup has a complete lattice structure. In fact, we prove that for any $\CatCa$-algebra of stable rank one $A$, the assignment $I\longmapsto\Cu_1(I)$ defines a complete lattice isomorphism between the set of ideals of $A$ and the set of ideals of its unitary Cuntz semigroup $\Cu_1(A)$. Further, we introduce a notion of quotients and exactness for the (non abelian) category $\Cu^\sim$. We show that $\Cu_1(A)/\Cu_1(I)\simeq \Cu_1(A/I)$ for any ideal $I$ in $A$ and that the functor $\Cu_1$ is exact. Finally, we link a $\Cu^\sim$-semigroup with the $\Cu$-semigroup of its positive elements and the abelian group of its maximal elements in a split-exact sequence. This result allows us to extract additional information that lies within the unitary Cuntz semigroup of a $\CatCa$-algebra of stable rank one.
\end{abstract}
\maketitle

%%%%%%%%% Intro
\section{Introduction}
In the last decade, the Cuntz semigroup has emerged as a suitable invariant in the classification of non-simple $\CatCa$-algebras. It is now well-established that this positively ordered monoid is a continuous functor from the category of $\CatCa$-algebras to the category of abstract Cuntz semigroups, written $\Cu$ (see \cite{CEI08} and \cite{APT14}). Moreover, an abstract notion of ideals and quotients in the category $\Cu$ has been considered in \cite{CRS10} and it has been proved that the Cuntz semigroup nicely captures the lattice of ideals of a $\CatCa$-algebra $A$, that we write $\Lat(A)$. In fact, for any $\CatCa$-algebra, the assignment $I\longmapsto \Cu(I)$ defines a complete lattice isomorphism between $\Lat(A)$ and the set of ideals of $\Cu(A)$, that we write $\Lat(\Cu(A))$ (see \cite[\S 5.1.6]{APT14}). These results make the Cuntz semigroup a valuable asset whenever considering non-simple $\CatCa$-algebras. 
While the Cuntz semigroup has already provided notable results for classification (see e.g. \cite{RobNCCW1}, \cite{RS09}), one often has to restrict itself to the case of trivial $\K_1$ since the Cuntz semigroup fails to capture the $\K_1$-group information of a $\CatCa$-algebra.
To address this issue, the author has introduced a unitary version of the Cuntz semigroup for $\CatCa$-algebras of stable rank one, written $\Cu_1$ (see \cite{C20one}). This invariant, built from pairs of positive and unitary elements, resembles the construction of the Cuntz semigroup and defines a continuous functor from the category of $\CatCa$-algebra of stable rank one to the category $\Cu^\sim$ of (not necessarily positively) ordered monoids satisfying the order-theoretic axioms (O1)-(O4) introduced in \cite{CEI08}. 

In this paper, we investigate further this new construction and we affirmatively answer the question whether this unitary version of the Cuntz semigroup also captures the lattice of ideals of a $\CatCa$-algebra of stable rank one. We specify that the category $\Cu^\sim$ does not require the underlying monoids to be positively ordered, which hinders the task to generalize notions introduced in the category $\Cu$. For instance,  we cannot characterize a $\Cu^\sim$-ideal of a countably-based $\Cu^\sim$-semigroup by its largest element, as is done for countably-based $\Cu$-semigroups, since such an element might not exist in general. As a result, two axioms, respectively named (PD), for \emph{positively directed} and (PC), for \emph{positively convex} appear as far as the definition of a $\Cu^\sim$-ideal is concerned. The axiom (PD) has already been introduced in \cite{C20one}, where the author has established that any positively directed $\Cu^\sim$-semigroup $S$ either has maximal elements forming an absorbing abelian group, termed $S_{max}$, or else has no maximal elements. We finally point out that any $\Cu$-semigroup $S$ satisfies these axioms and that the generalization of a $\Cu^\sim$-ideal matches with the usual definition of a $\Cu$-ideal for any $\Cu$-semigroup $S$. In the course of this investigation, we also show that the functor $\Cu_1$ satisfies expected properties regarding ideals, quotients and exact sequences. These results help us to dig in depth the functorial relations between $\Cu,\K_1$ and $\Cu_1$ found in \cite[\S 5]{C20one}. 

More concretely, this paper shows that the set of $\Cu^\sim$-ideals of such a $\Cu^\sim$-semigroup $S$ is a complete lattice naturally isomorphic to the complete lattice of $\Cu$-ideals of its positive cone $S_+$. Furthermore, we prove that:

\begin{thm}
For any $\CatCa$-algebra $A$ of stable rank one, the unitary Cuntz semigroup $\Cu_1(A)$ is positively directed and positively convex.

Moreover, the assignment $I\longmapsto \Cu_1(I)$ defines a complete lattice isomorphism between $\Lat(A)$ and $\Lat(\Cu_1(A))$ that maps the sublattice $\Lat_f(A)$ of ideals in $A$ that contain a full, positive element onto the sublattice $\Lat_f(\Cu_1(A))$ of ideals in $\Cu_1(A)$ that are singly-generated by a positive element.
In particular, $I$ is simple if and only if $\Cu_1(I)$ is simple.
\end{thm}

\begin{thm}
Let $A$ be a $\CatCa$-algebra of stable rank one and let $I\in\Lat(A)$. Consider the canonical short exact sequence: $0 \longrightarrow I\overset{i}\longrightarrow A\overset{\pi}\longrightarrow A/I\longrightarrow 0 $. Then:

(i) $\Cu_1(\pi)$ induces a $\Cu^\sim$-isomorphism $\Cu_1(A)/\Cu_1(I)\simeq \Cu_1(A/I)$.

(ii) The following sequence is short exact in $\Cu^\sim$:
\[
\xymatrix{
0\ar[r]^{} & \Cu_1(I)\ar[r]^{i^*} & \Cu_1(A)\ar[r]^{\pi^*} & \Cu_1(A/I)\ar[r]^{} & 0
} 
\]
\end{thm}

\begin{thm}
Let $S$ be a positively directed $\Cu^\sim$-semigroup that has maximal elements. Then the following sequence in $\Cu^\sim$ is split-exact: 
\[
\xymatrix{
0\ar[r]^{} & S_+\ar[r]^{i} & S\ar[r]^{j} & S_{max}\ar@{.>}@/_{-1pc}/[l]^{q}\ar[r]^{} & 0
} 
\]
where $i$ is the canonical injection, $j(s):=s+e_{S_{max}}$ and $q(s):=s$.
\end{thm}

The paper is organized as follows: In the first part, we define an abstract notion of a $\Cu^\sim$-ideal for any positively directed $\Cu^\sim$-semigroup. We then see that the smallest ideal containing an element might not always exist since the intersection of two $\Cu^\sim$-ideals is not necessarily a $\Cu^\sim$-ideal. However, the smallest ideal  containing an element $s$ of a positively directed and positively convex $\Cu^\sim$-semigroup $S$, where the notion of \emph{positively convex} is to be specified, always exists and is explicitly computed. We finally build a complete lattice structure on the set of $\Cu^\sim$-ideals of a positively directed and positively convex $\Cu^\sim$-semigroup $S$, relying on the natural set bijection between $\Lat(S)\simeq \Lat(S_+)$, where $S_+\in\Cu$ is the positive cone of $S$.

We also study the notion of quotients and exactness in the category $\Cu^\sim$. Among others, we show that a quotient  of a positively directed and positively ordered $\Cu^\sim$-semigroup by an ideal is again a positively directed and positively ordered $\Cu^\sim$-semigroup. Moreover, the functor $\Cu_1$ preserves quotients and short exact sequence of ideals. We finally use the split-exact sequence $0\longrightarrow S_+\longrightarrow S\longrightarrow S_{max}\longrightarrow 0$ described above to unravel commutative diagrams with exact rows linking $\Cu,\K_1$ and $\Cu_1$ of a separable $\CatCa$-algebra with stable rank one -and its ideals-.

Note that this paper is the second part of a twofold work (following up \cite{C20one}) and completes the properties of the unitary Cuntz semigroup established during the author's PhD thesis. We also mention that the unitary Cuntz semigroup -through these results- will be used in a forthcoming paper to distinguish two non-simple unital separable $\CatCa$-algebras with stable rank one, which originally agree on $\K$-Theory and the Cuntz semigroup; see \cite{C21}.

\textbf{Acknowledgments} The author would like to thank Ramon Antoine for suggesting a more adequate version of the \textquoteleft positively convex\textquoteright\ property, and both Ramon Antoine and Francesc Perera for insightful comments about the paper. The author also thanks the referee for his/her pertinent comments that have helped to reformulate some part of the manuscript in a better way.

%%%% Preliminaries
\section{Preliminaries}

We use $\CatoM$ to denote the category of ordered monoids, in contrast to the category of positively ordered monoids, that we write $\CatPoM$. We also use $\CatCa_{\sr1}$ to denote the full subcategory of $\CatCa$-algebras of stable rank one.

\subsection{The Cuntz semigroup}
We recall some definitions and properties on the Cuntz semigroup of a $\CatCa$-algebra. More details can be found in \cite{APT14}, \cite{APT09}, \cite{CEI08}, \cite{T19}.

\begin{prg}\textbf{(The Cuntz semigroup of a $\CatCa$-algebra.)}
\label{dfn:Cusg}
Let $A$ be a $\CatCa$-algebra. We denote by $A_+$ the set of positive elements. Let $a$ and $b$ be in $A_+$. We say that $a$ is Cuntz subequivalent to $b$, and we write $a\lesssim_{\Cu} b$, if there exists a sequence $(x_n)_{n\in\N}$ in $A$ such that $a=\lim\limits_{n\in\N}x_nbx_n^*$. After antisymmetrizing this relation, we get an equivalence relation over $A_+$, called Cuntz equivalence, denoted by $\sim_{\Cu}$.

Let us write $\Cu(A):= (A\otimes\mathcal{K})_+/\!\!\sim_{\Cu}$, that is, the set of Cuntz equivalence classes of positive elements of $A\otimes\mathcal{K}$. Given $a\in (A\otimes\mathcal{K})_+ $, we write $[a]$ for the Cuntz class of $a$. The set $\Cu(A)$ is equipped with an addition as follows: let $v_1$ and $v_2$ be two isometries in the multiplier algebra of $A\otimes\mathcal{K}$, such that $v_1v_1^*+v_2v_2^*=1_{M(A\otimes\mathcal{K})}$. Consider the $^*$-isomorphism $\psi: M_2(A\otimes\mathcal{K})\longrightarrow A\otimes\mathcal{K}$ given by $\psi(\begin{smallmatrix} a & 0\\ 0 & b \end{smallmatrix})=v_1av_1^*+v_2bv_2^*$, and we write $a \oplus b:=\psi(\begin{smallmatrix} a & 0\\ 0 & b \end{smallmatrix})$. 
For any $[a],[b]$ in $\Cu(A)$, we define $[a]+[b]:=[a\oplus b]$ and $[a]\leq [b]$ whenever $a\lesssim_{\Cu} b$. In this way $\Cu(A)$ is a partially ordered semigroup called \emph{the Cuntz semigroup of $A$}. 

For any $^*$-homomorphism $\phi:A\longrightarrow B$, one can define $\Cu(\phi):\Cu(A)\longrightarrow\Cu(B)$, a semigroup map, by $[a]\longmapsto [(\phi\otimes id_\mathcal{K})(a)]$. Hence, we get a functor from the category of $\CatCa$-algebras into a certain subcategory of $\CatPoM$, called the category $\Cu$, that we describe next. 
\end{prg}

\begin{prg}\textbf{(The category $\Cu$)}.
\label{dfn:llCU}
Let $(S,\leq)$ be a positively ordered semigroup and let $x,y$ in $S$. We say that $x$ is \emph{way-below} $y$ and we write $x\ll y$ if, for all increasing sequences $(z_n)_{n\in\N}$ in $S$ that have a supremum, if $\sup\limits_{n\in\N} z_n\geq y$, then there exists $k$ such that $z_k\geq x$. This is an auxiliary relation on $S$ called the \emph{way-below relation} or the \emph{compact-containment relation}. In particular $x\ll y$ implies $x\leq y$ and we say that $x$ is a \emph{compact element} whenever $x\ll x$. 

We say that $S$ is an abstract Cuntz semigroup, or a $\Cu$-semigroup, if it satisfies the following order-theoretic axioms: 

$\,\,$(O1): Every increasing sequence of elements in $S$ has a supremum. 

$\,\,$(O2): For any $x\in S$, there exists a $\ll$-increasing sequence $(x_n)_{n\in\N}$ in $S$ such that $\sup\limits_{n\in\N} x_n= x$.

$\,\,$(O3): Addition and the compact containment relation are compatible.

$\,\,$(O4): Addition and suprema of increasing sequences are compatible.

A \emph{$\Cu$-morphism} between two $\Cu$-semigroups $S,T$ is a positively ordered monoid morphism that preserves the compact containment relation and suprema of increasing sequences. 

The category of abstract Cuntz semigroups, written $\Cu$, is the subcategory of $\CatPoM$ whose objects are $\Cu$-semigroups and morphisms are $\Cu$-morphisms. 
\end{prg}

\begin{prg}\textbf{(Countably-based $\Cu$-semigroups.)}
\label{dfn:Cuorderunit}
Let $S$ be a $\Cu$-semigroup. We say that $S$ is \emph{countably-based} if there exists a countable subset $B\subseteq S$ such that for any $a,a'\in S$ such that $a'\ll a$, then there exists $b\in B$ such that $a'\leq b \ll a$. The set $B$ is often referred to as a \emph{basis}.
An element $u\in S$ is called an \emph{order-unit} of $S$ if for any $x\in S$, there exists $n\in\overline{\N}$ such that $x\leq nu$, where $\overline{\N}:=\N\sqcup\{\infty\}$.

Let $S$ be a countably-based $\Cu$-semigroup. Then $S$ has a maximal element, or equivalently, it is singly-generated. Let us also mention that if $A$ is a separable $\CatCa$-algebra, then $\Cu(A)$ is countably-based. In fact, its largest element, that we write $\infty_A$, can be explicitly constructed as follows: Let $s_A$ be any strictly positive element (or full positive) in $A$. Then $\infty_{A}=\sup\limits_{n\in\N}n[s_A]$. A fortiori, $[s_A]$ is an order-unit of $\Cu(A)$.
\end{prg}

\begin{prg}\textbf{(Lattice of ideals in $\Cu$.)}
\label{prg:latticecu}
Let $S$ be a $\Cu$-semigroup. An \emph{ideal} of $S$ is a submonoid $I$ that is closed under suprema of increasing sequences and such that for any $x,y$ such that $x\leq y$ and $y\in I$, then $x\in I$. 

It is shown in \cite[\S 5.1.6]{APT14}, that for any $I,J$ ideals of $S$, $I\cap J$ is again an ideal. Therefore for any $x\in S$, the ideal generated by $x$, defined as the smallest ideal of $S$ containing $x$, and written $I_x$, is exactly the intersection of all ideals of $S$ containing $x$. An explicit computation gives us $I_x:=\{y\in S \mid y\leq \infty x\}$. 

Moreover it is shown that $I+J:=\{z\in S\mid z\leq x+y, x\in I, y\in J\}$ is also an ideal. Thus we write $\Lat(S):=\{$ideals of $S\}$, which is a complete lattice under the following operations: for any two  $I,J\in \Lat(S)$, we define $I\wedge J:= I\cap J$ and $I\vee J:=I+J$. 

Furthermore, for any $\CatCa$-algebra $A$, we have that $\Cu(I)$ is an ideal of $\Cu(A)$ for any $I\in\Lat(A)$. In fact, we have a lattice isomorphism as follows:
\[
	\begin{array}{ll}
		\Lat(A)\overset{\simeq}{\longrightarrow} \Lat(\Cu(A))\\
		\hspace{0,85cm} I\longmapsto \Cu(I)
	\end{array}
\]
Finally, whenever $S$ is countably-based, any ideal $I$ of $S$ is singly-generated, for instance by its largest element, that we also write $\infty_I$. In particular, for any $\CatCa$-algebra $A$, any $a,b\in (A\otimes\mathcal{K})_+$, if $[a]\leq [b]$ in $\Cu(A)$, then $I_a\subseteq I_b$, or equivalently $I_{[a]}\subseteq I_{[b]}$. (The converse is a priori not true: $I_{x}=I_{kx}$ for any $x\in\Cu(A)$, any $k\in\overline{\N}$ but in general $x\neq kx$.)
\end{prg}

\begin{prg}\textbf{(Quotients in $\Cu$.)}
\label{prg:quotientcu}
Let $S$ be a $\Cu$-semigroup and $I\in\Lat(S)$. Let $x,y\in S$. We write $x\leq_I y$ if: there exists $z\in I$ such that $x\leq z+y$. By antisymmetrizing $\leq_I$, we obtain an equivalence relation $\sim_I$ on $S$. Define $S/I:= S/\!\!\sim_I$. For $x\in S$, write $\overline{x}:=[x]_{\sim_I}$ and equip $S/I$ with the following addition and order: Let $x,y\in S$. Then $\overline{x}+\overline{y}:=\overline{x+y}$ and $\overline{x}\leq\overline{y}$, if $x\leq_I y$. These are well-defined and $(S/I,+,\leq)$ is a $\Cu$-semigroup, often referred to as the \emph{quotient of $S$ by $I$}. Moreover, the canonical quotient map $S\longrightarrow S/I$ is a surjective $\Cu$-morphism. Finally, for any $\CatCa$-algebra $A$ and any $I\in\Lat(A)$, we have $\Cu(A/I)\simeq \Cu(A)/\Cu(I)$; see \cite[Corollary 2]{CRS10}.
\end{prg}

\subsection{The unitary Cuntz semigroup}
We recall some definitions and properties on the $\Cu_1$-semigroup of a $\CatCa$-algebra with stable rank one. More details can be found in \cite{C20one}.

\begin{prg}\textbf{(The unitary Cuntz semigroup of a $\CatCa$-algebra - The category $\Cu^\sim$.)}
\label{dfn:Cu1sg}
Let $A$ be a $\CatCa$-algebra of stable rank one, let $a,b\in A_+$ such that $a\lesssim_{\Cu} b$. Using the stable rank one hypothesis, there exist \emph{standard morphisms} $\theta_{ab}:\her (a)^\sim\hooklongrightarrow \her (b)^\sim$ such that $[\theta_{ab}(u)]_{\K_1}$ does not depend on the standard morphism chosen, for any unitary element $u\in \her (a)^\sim$. That is, there is a canonical way (up to homotopy equivalence) to extend unitary elements of $\her (a)^\sim$ into unitary elements of $\her (b)^\sim$ . Now, let $u,v$ be unitary elements of $\her (a)^\sim,\her (b)^\sim$ respectively. We say that $(a,u)$ is \emph{unitarily Cuntz subequivalent} to $(b,v)$, and we write $(a,u)\lesssim_1 (b,v)$, if $a\lesssim_{\Cu} b$ and $\theta_{ab}(u)\sim_h v$. After antisymmetrizing this relation, we get an equivalence relation on $H(A):=\{(a,u)\mid a\in (A\otimes\mathcal{K})_+ , u\in \mathcal{U}(\her (a)^\sim)\}$, called \emph{the unitary Cuntz equivalence}, denoted by $\sim_1$.

Let us write $\Cu_1(A):= H(A)/\!\!\sim_{1}$. The set $\Cu_1(A)$ can be equipped with a natural order given by $[(a,u)]\leq [(b,v)]$ whenever $(a,u)\lesssim_{1} (b,v)$, and we set $[(a,u)]+[(b,v)]:=[(a\oplus b,u\oplus v)]$. In this way $\Cu_1(A)$ is a semigroup called \emph{the unitary Cuntz semigroup of $A$}. 

Any $^*$-homomorphism $\phi:A\longrightarrow B$ naturally induces a semigroup morphism $\Cu_1(\phi):\Cu_1(A)\longrightarrow\Cu_1(B)$, by sending $[(a,u)]\longmapsto [(\phi\otimes id_\mathcal{K})(a),(\phi\otimes id_\mathcal{K})^\sim(u)]$. Hence, we get a functor from the category of $\CatCa$-algebras of stable rank one into a certain subcategory of ordered monoids, denoted by $\CatoM$, called the category $\Cu^\sim$, that we describe in the sequel. 

Let $(S,\leq)$ be an ordered monoid. Recall the compact-containment relation defined in \autoref{dfn:llCU}. We say that $S$ is a $\Cu^\sim$-semigroup if $S$ satisfies axioms (O1)-(O4) and $0\ll 0$. We emphasize that we do not require the monoid to be positively ordered. A \emph{$\Cu^\sim$-morphism} between two $\Cu^\sim$-semigroups $S,T$ is an ordered monoid morphism that preserves the compact-containment relation and suprema of increasing sequences. 

The category of abstract unitary Cuntz semigroups, written $\Cu^\sim$, is the subcategory of $\CatoM$ whose objects are $\Cu^\sim$-semigroups and morphisms are $\Cu^\sim$-morphisms. Actually, as shown in \cite[Corollary 3.21]{C20one}, the functor $\Cu_1$ from the category $\CatCa_{\sr1}$ to the category $\Cu^\sim$ is arbitrarily continuous.
\end{prg}

\begin{prg}\textbf{(Alternative picture of the $\Cu_1$-semigroup.)}
\label{prg:newpicture}
We will sometimes use an alternative picture described in \cite[\S 4.1]{C20one}.
First, recall that for a $\CatCa$-algebra $A$, $\Lat_f(A)$ is the sublattice of $\Lat(A)$ consisting of ideals that contain a full, positive element. Also recall that $\{\sigma\text{-unital ideals of }A\}\subseteq \Lat_f(A)$ and if moreover $A$ is separable, then the converse inclusion holds. Finally, for any $I\in\Lat_f(A)$, we define $\Cu_f(I):=\{x\in \Cu(A) \mid I_x=\Cu(I)\}$ to be the set of full elements in $\Cu(I)$. 

Let $A$ be a $\CatCa$-algebra of stable rank one such that $\Lat_f(A)=\{\sigma\text{-unital ideals of }A\}$.
Then $\Cu_1(A)$ can be pictured as \vspace{-0,3cm}\[\bigsqcup\limits_{I\in\Lat_f(A)} \Cu_f(I)\times \K_1(I)\] that we also write $\Cu_1(A)$. The addition and order are defined as follows: For any $(x,k),(y,l)\in \Cu_1(A)$
\[
\left\{
\begin{array}{ll}
(x,k)\leq (y,l) \text{ if: } x\leq y \text{ and } \delta_{I_xI_{y}}(k)=l.\\
(x,k)+(y,l)=(x+y,\delta_{I_xI_{x+y}}(k)+\delta_{I_yI_{x+y}}(l)).
\end{array}
\right.
\]
where $\delta_{IJ}:=\K_1(I\overset{i}\hooklongrightarrow J)$, for any $I,J\in\Lat_f(A)$ such that $I\subseteq J$.

Let $A,B$ be $\CatCa$-algebras of stable rank one and let $\phi:A\longrightarrow B$ be a $^*$-homomorphism. For any $I\in\Lat_f(A)$, we write $J:=\overline{B\phi(I)B}$, the smallest ideal of $B$ that contains $\phi(I)$. Then $J\in\Lat_f(B)$ and $\Cu_1(\phi)$ can be rewritten as $(\Cu(\phi),\{\K_1(\phi_{|I})\}_{I\in\Lat_f(A)})$, where $\phi_{|I}:I\longrightarrow J$. Observe that we might write $\alpha,\alpha_0, \alpha_I$ to denote $\Cu_1(\phi),\Cu(\phi),\K_1(\phi_{|I})$ respectively.
\end{prg}

%%%%%%%%%%%IDEALS IN CU SIM
\section{Ideal structure in the category \texorpdfstring{$\Cu^\sim$}{Cu-tilde}}
\label{sec:idealstructure}

In this section we define and study the notion of ideals in the category $\Cu^\sim$. Since the underlying monoid of a $\Cu^\sim$-semigroup might not be positively ordered, definitions and results of the category $\Cu$ cannot be applied and some extra work is needed. When it comes to a concrete $\Cu^\sim$-semigroup, -that is, coming from a $\CatCa$-algebra of stable rank one $A$- we wish that a $\Cu^\sim$-ideal satisfies natural properties, e.g. $\Cu_1(I)$ is an ideal of $\Cu_1(A)$ or $\Lat(A)$ is entirely captured by the set of $\Cu^\sim$-ideals of $\Cu_1(A)$. For that matter, we first have to study the set of maximal elements of a $\Cu^\sim$-semigroup. We show that under additional axioms -satisfied by any $\Cu_1(A)$-, namely the axioms (PD) and (PC), the set maximal elements of a $\Cu^\sim$-semigroup forms, when not empty, an absorbing abelian group. From there, we are able to define a suitable notion of $\Cu^\sim$-ideal. We will also use concepts from Domain Theory that we recall now (see \cite{GHKLMS03}).

Finally, we say that a $\Cu^\sim$-semigroup $S$ is \emph{countably-based} if there exists a countable subset $B\subseteq S$ such that for any pair $a'\ll a$, there exists $b\in B$ such that $a'\leq b \ll a$.

\subsection{Definition of a \texorpdfstring{$\Cu^\sim$}{Cu-tilde} ideal}
\begin{dfn}\label{dfn:scottopen} \cite[Definition II.1.3]{GHKLMS03}
Let $S$ be a $\Cu^\sim$-semigroup. A subset $O\subseteq S$ is \emph{Scott-open} if:

(i) $O$ is an upper set, that is, for any $y\in S$, $y\geq x\in O$ implies $y\in O$.

(ii) For any $x\in O$, there exists $x'\ll x$ such that $x'\in O$. Equivalently, for any increasing sequence of $S$whose supremum belongs to $O$, there exists an element of the sequence also in $O$.

Dually we say that $F\subseteq S$ is \emph{Scott-closed} if $S\setminus F$ is Scott-open, that is, if it is a lower set that is closed under suprema of increasing sequences.
\end{dfn}

Let us check the equivalence of (ii) in the above definition: Let $O$ be an upper set of $S$ and let $x\in O$. Suppose there exists $x'\ll x$ such that $x'\in O$. Let $(x_n)_n$ be any increasing sequence whose supremum is $x$. By definition of $\ll$, there exists $x_n\geq x'$, hence $x_n$ is also in $O$. 
Conversely, using (O2), there exists a $\ll$-increasing sequence $(x_n)_n$ whose supremum is $x$. By hypothesis, there exists $n$ such that $x_n\in O$, and by construction $x_n\ll x$. This finishes the proof.

\begin{dfn}
Let $S$ be a $\Cu^\sim$-semigroup. We define the following axioms: 

(PD): We say that $S$ is \emph{positively directed} if, for any $x\in S$, there exists $p_x\in S$ such that $x+p_x\geq 0$.

(PC): We say that $S$ is \emph{positively convex} if, for any $x,y\in S$ such that $y\geq 0$ and $x\leq y$, we have $x+y\geq 0$.
\end{dfn}

The axiom (PC) ensures that the only negative element of $S$ is $0$, while the axiom (PD) ensures that any non-positive element has a \textquoteleft symmetric\textquoteright\, \!such that their sum is a positive element. Furthermore, the set of maximal elements of a positively directed $\Cu^\sim$-semigroup has an abelian group structure (see \cite[\S 5.1]{C20one}). We first show that these axioms are satisfied by any concrete $\Cu^\sim$-semigroup.

\begin{lma}
Let $A$ be a $\CatCa$-algebra of stable rank one. Then $\Cu_1(A)$ is positively directed and positively convex.
\end{lma}

\begin{proof}
Let $A$ be a $\CatCa$-algebra of stable rank one and consider $[(a,u)]\in\Cu_1(A)$, where $a\in(A\otimes\mathcal{K})_+$ and $u\in \mathcal{U}(\her(a)^\sim)$. Observe that $[(a,u)]+[(a,u^*)]=[(a\oplus a,1)]\geq 0$, and so $\Cu_1(A)$ is positively directed. Now let $[(b,1)]$ be a positive element in $\Cu_1(A)$ such that $[(a,u)]\leq [(b,1)]$. Since $[(a,u)]\leq [(b,1)] $, we know that $\chi_{ab}([u])=[1]$. Therefore, $\chi_{a(a\oplus b)}([u])=[1]$, and we deduce that $[(a,u)]+[(b,1)] = [(a\oplus b,1)]$ is a positive element in $\Cu_1(A)$, which finishes the proof.
\end{proof}

\begin{dfn}
Let $S$ be a $\Cu^\sim$-semigroup. We define
$S_{max}:=\{x\in S \mid \text{ if } y\geq x, \text{ then } y=x\}$, the set of maximal elements of $S$.
\end{dfn}

\begin{prop}\cite[Proposition 5.4]{C20one}
\label{prop:PCABGP}
Let $S$ be a positively directed $\Cu^\sim$-semigroup. Then $S_{max}$ is either empty or an absorbing abelian group in $S$ whose neutral element $e_{S_{\max}}$ is positive. 
\end{prop}

\begin{rmk}
Whenever $S$ is a positively directed $\Cu^\sim$-semigroup that has maximal elements, then $e_{S_{max}}$ is the only positive element of $S_{max}$ or, equivalently, the only positive maximal element of $S$. Also, we mention that any countably-based $\Cu^\sim$-semigroup has maximal elements.
\end{rmk}

\begin{lma}
\label{lma:PCequi}
Let $S$ be $\Cu^\sim$-semigroup that has maximal elements. Then the following are equivalent: 

(i) $S$ is positively directed.

(ii) For any $x\in S$, there exists a unique $p_x\in S_{max}$ such that $x+p_x\geq 0$. 

(iii) $S_{max}$ is an absorbing abelian group in $S$ whose neutral element $e_{S_{\max}}$ is positive.
\end{lma}

\begin{proof}
That (ii) implies (i) is clear. That (i) implies (iii) is proved in \cite[Proposition 5.4]{C20one}.

Let us show now that (iii) implies (ii): Let $x\in S$ and write $e:=e_{S_{max}}$. Let $q:=x+e$. Note that $q$ belongs to $S_{max}$ by (iii). Denote by  $p_x$ the inverse of $q$ in $S_{max}$. We have $x+e+p_x=e$, and $x+p_x\in S_{max}$ by assumption. Therefore $x+p_x+e=x+p_x=e\geq 0$. Now suppose there exists another $r\in S_{max}$ such that $r+x\geq 0$. Then $r+x+p_x=p_x$. However $x+p_x=e$, hence $r=p_x$, which ends the proof.
\end{proof}

Notice that for a $\Cu$-semigroup $S$, we have that $S_{max}$ is either empty, or the trivial group consisting of the largest element of $S$.
Furthermore, the axioms (PD) and (PC) can be defined for ordered monoids and all the proofs above hold. We now define the notion of a \emph{positively stable submonoid} for positively directed $\Cu^\sim$-semigroup that will lead to the definition of a $\Cu^\sim$-ideal.

\begin{dfn}
Let $S$ be a positively directed $\Cu^\sim$-semigroup. Let $M$ be a submonoid of $S$. We say $M$ is \emph{positively stable} if it satisfies the following:

(i) $M$ is a positively directed ordered monoid.

(ii) For any $x\in S$, if $(x+P_x)\bigcap M\neq \emptyset$ then $x\in M$, where $P_x:=\{y\in S \mid x+y\geq 0\}$. 
\end{dfn}

Axiom (PD) ensures that $P_x\neq \emptyset$. In fact, $P_x$ is a Scott-open set in $S$ (so is $x+P_x$): $P_x$ is clearly an upper set and using $0\ll 0$ and (O2), and one can check that $P_x$ satisfies (ii) of \autoref{dfn:scottopen}. In particular, $S_+=P_0$ is Scott-open in $S$.
 
\begin{dfn}
\label{dfn:orderideal}
Let $S$ be a positively directed $\Cu^\sim$-semigroup. We say that $I\subseteq S$ is an \emph{order-ideal} (or \emph{ideal}) of $S$ if $I$ is a Scott-closed, positively stable submonoid of $S$. 
 
We say that $S$ is \emph{simple} if it only contains the trivial ideals $\{0\}$ and $S$. 
\end{dfn}

It is for the reader to check that any ideal $I$ of a (positively directed) $\Cu^\sim$-semigroup $S$ is a positively directed $\Cu^\sim$-semi\-group. Moreover, if $S$ is positively convex, then so is $I$. Finally, $I$ continuously order-embeds into $S$ (that is, the canonical inclusion $i:I\lhook\joinrel\longrightarrow S$ is a Scott-continuous order-embedding). 

We naturally want to define the ideal generated by an element. However, we cannot ensure that the intersection of ideals is still an ideal. In fact, being positively directed is not preserved under intersection, thus we define the ideal generated by an element abstractly as follows:

\begin{dfn}
Given $x\in S$, we define $\Idl(x)$ as the smallest ideal of $S$ containing $x$, that is, $x\in \Idl(x)$ and for any $J$ ideal of $S$ containing $x$ we have $J\supseteq \Idl(x)$. Note that this ideal might not exist.
\end{dfn}

Here, we offer an example of two ideals of a countably-based positively directed and positively convex $\Cu^\sim$-semigroup, whose intersection fails to be positively directed, and hence fails to be an ideal:

Let $S$ be the subset of ${\overline{\N}}^3\times\Z$ defined as follows:
\[
S:=\{((n_1,n_2,n_3),k)\in \overline{\N}^3\times\Z \mid k\geq 0, \text{ if }n_1=n_3=0, \text{ and }k=0, \text{ if }n_1=n_2=n_3=0\}.
\] 
We put on this set a component-wise sum and we define for any two pairs: $(g,k)\leq (h,l)$ if $g\leq h $ in ${\overline{\N}}^3 $ and $k=l$ in $\Z$. Notice that $S_+= {\overline{\N}}^3 \times \{0\}$. One can check that $(S,+,\leq)$ is a countably-based positively directed and positively convex $\Cu^\sim$-semigroup. 

Now consider $I_1:=((\overline{\N}\times \overline{\N}\times\{0\})\times\Z)\cap S$ and $I_2:=((\{0\}\times \overline{\N}\times \overline{\N} )\times\Z)\cap S$. Again, one can check that those are ideals of $S$ as defined earlier. However, $I_1\cap I_2 = ((\{0\}\times\overline{\N}_*\times\{0\})\times \Z_+)\sqcup \{0_S\}$ is not positively directed. Indeed, let $x:=((0,n,0),1)\in I_1\cap I_2$. Observe that any element $y\in I_1\cap I_2 $ is of the form $((0,n,0),k)$ for some $n\in\N$ and $k\geq 0$. Thus, there is no $y\in I_1\cap I_2 $ such that $x+y\geq 0$ and hence $I_1\cap I_2$ is not positively directed.

\begin{prop}
\label{prop:existenceideal}
Let $S$ be positively directed and positively convex $\Cu^\sim$-semigroup. Let $x$ be a positive element of $S$. Then $\Idl(x)$ exists and we have the following:
\[
Idl(x)=\{y\in S \mid \text{ there is } y'\in S \text{ with } 0\leq y+y'\leq\infty x\}
\vspace{-0cm}\]
\end{prop}

\begin{proof}
Let us define $I_x:=\{y\in S \mid \text{ there is } y'\in S \text{ with } 0\leq y+y'\leq\infty x\}$. We want to prove that $I_x$ is the $\Cu^\sim$-ideal generated by $x$.

First, we show that $I_x$ is a submonoid of $S$ that contains $x$. Using (O1), we know that $\infty x:=\sup\limits_{n\in \N}nx$ is a positive element. Moreover $0\leq 0+0\leq\infty x$, hence $0\in I_x$. We also know that for any $n,m$ in $\overline{\N}$, $0\leq nx+mx\leq\infty x$. So we get that $\{nx \mid n\in\overline{\N}\}\subseteq I_x$. Let $y_1,y_2$ in $I_x $. Then one easily checks that $0\leq (y_1+y_2)+(y_1'+y_2')\leq 2 (\infty x)=\infty x$, hence $I_x $ is closed under addition. This proves it is a submonoid of $S$ that contains $x$. 

\emph{Claim}: $\infty x$ is a maximal positive element of $I_x$ (in fact, the unique maximal positive element of $I_x$). Let $y\in I_x$ such that $y\geq 0$. There exists $y'\in I_x$ such that $0\leq y+y'\leq \infty x$. Since $y\geq 0$, we get that $y'\leq y+y'\leq \infty x$. So, by axiom (PC) we deduce that $0\leq y'+\infty x$. Now we add $y$ on both sides to get that $y\leq y+y'+\infty x\leq 2(\infty x)=\infty x$. Therefore, for any positive element $y$ of $I_x$, $y\leq \infty x$, which proves the claim.

Let us now prove that $I_x$ is closed under suprema of increasing sequences. Let $(y_n)_n$ be an increasing sequence in $I_x$ and let $y$ be its supremum in $S$. Let $y_0'$ be such that $0\leq y_0+y_0'\leq \infty x$, where $y_0$ is the first term of $(y_n)_n$. Observe that $y_0'$ belongs to $I_x$. Since $I_x$ is closed under addition, for any $n\in \N$, we have $y_n+y_0'\in I_x$. Therefore we can choose $z_n\in I_x$ such that   $(0\leq)\,\, y_n+y_0'+z_n\leq \infty x$. Finally choose $z_n'\in I_x$ such that $0\leq z_n+z'_n\leq \infty x$. 

Thus, we have on the one hand that $0\leq y_n+y_0'\leq(y_n+y_0')+(z_n+z_n')$ and on the other hand that $(y_n+y_0'+z_n)+z_n'\leq\infty x+z_n'$ for any $n\in \N$. Now since $I_x$ is submonoid of $S$ that contains $x$ and $z_n'\in I_x$, we get that $\infty x+z_n'$ is a positive element of $I_x$. Now since $\infty x=2(\infty x)$, we have $(\infty x+z_n')=2(\infty x)+z_n'\geq \infty x $. By maximality of $\infty x$ in $I_x$, we get that $0\leq y_n+y_0'\leq \infty x$, for any $n\in\N$. Using axioms (O3) and (O4), we pass to suprema and we obtain $0\leq y+y_0'\leq \infty x$, that is, $y\in I_x$. So $I_x$ is closed under suprema of increasing sequences.

We also have to show that $I_x$ is positively stable. Take any $z\in S$ such that there exists $z'$ with $0\leq z+z'$ and $(z+z')\in I_x$. We know there is a $y\in I_x $ such that $0\leq z+z'+y\leq\infty x$. Hence $z\in I_x$.

Next, we check that $I_x$ is a lower set. Let $z\leq y$ with $y\in I_x$. We know that there exists $y'\in I_x$ such that $0\leq y+y'\leq \infty x$. Since $z+y'\leq y+y'$, we deduce by axiom (PC) that $0\leq z+y'+y+y'\leq 2(y+y')\leq \infty x$. Therefore $z\in I_x$, that is, $I_x$ is a lower set, which ends the proof that $I_x$ is an ideal of $S$ containing $x$.

Lastly, let $J$ be an ideal of $S$ containing $x$. Then it contains $\infty x=e_{({I_x})_{max}}$. Thus if $y\in I_x$, we know that there exists $y'\in I_x$ such that $0\leq y+y'\leq \infty x$, and therefore $y+y'\in (y+P_y)\cap J$. Since $J$ is positively stable, this implies that $y\in J$. We obtain $J\supseteq I_x$, which gives us that $I_x$
is the ideal generated by $x$.
\end{proof}

We mention that a notion of ideals has been defined in \autoref{dfn:orderideal} for positively directed $\Cu^\sim$-semigroups. However, the existence of an ideal generated by a positive element requires the axiom (PC). Thus, from now on we only consider positively directed and positively convex $\Cu^\sim$-semigroups (and this will be specified). 
 
In the context of \autoref{prop:existenceideal}, observe that $I_x$ defined in the proof is equal to $\Idl(x)$ and from now on, we denote by $I_x$ the ideal generated by a positive element $x$. Also observe that $I_x$ is positively directed and has maximal elements. Thus, by \autoref{prop:PCABGP}, we know that $({I_x})_{max}$ is an absorbing abelian group whose neutral element is $\infty x$ corresponding to the unique maximal positive element of $I_x$. 

\begin{cor}
\label{cor:comparisonemax0}
Let $S$ be a positively directed and positively convex $\Cu^\sim$-semigroup. Let $I$ be an ideal of $S$. Then $I$ has maximal elements if and only if $I$ is singly-generated by a positive element, for instance by its (unique) maximal positive element $e_{I_{max}}$.
\end{cor}

\begin{proof}
If $I$ has maximal elements, then by \autoref{prop:PCABGP} we know that ${I}_{max}$ is an absorbing abelian group whose neutral element $e_{I_{max}}$ is the unique maximal positive element of $I$. Thus $I_{e_{I_{max}}}$ exists. Obviously,  $I_{e_{I_{max}}}\subseteq I$. Now let $x\in I$. Using \autoref{lma:PCequi}, we can find $y\in I_{max}$ such that $x+y=e_{I_{max}}$. Since $I$ is positively stable, we deduce that $x\in I_{e_{I_{max}}}$ and that $I_{e_{I_{max}}}=I$. Conversely, if $I$ is singly-generated by a positive element $x$, then from the proof of \autoref{prop:existenceideal}, we know that $\infty x \in I_{max}$ which ends the proof. 
\end{proof}

\subsection{Complete lattice of ideals} We now study the set of ideals of a positively directed and positively convex $\Cu^\sim$-semigroup $S$, that we denote $\Lat(S)$. We in fact show that $\Lat(S)$ has a natural structure of complete lattice and that, moreover, we have a lattice isomorphism between $\Lat(A)$ and $\Lat(\Cu_1(A))$ for any $\CatCa$-algebra $A$ of stable rank one. The sublattice $\Lat_f(S)$ consisting of ideals singly-generated by a positive element (or equivalently ideals that have maximal elements) will also be of an interest since the latter isomorphism maps $\Lat_f(A)$ onto $\Lat_f(\Cu_1(A))$ for any $\CatCa$-algebra $A$ of stable rank one.

\begin{prg}
Let $A$ be a $\CatCa$-algebra of stable rank one. Using the alternative picture of the unitary Cuntz semigroup of \autoref{prg:newpicture}, it is almost immediate that an element $(x,k)\in \Cu_1(A)$ belongs to $\Cu_1(I)$ if and only if $x\in \Cu(I)$. 
This allows us to prove in a similar fashion to the Cuntz semigroup (see \cite[Section 5.1]{APT14}) that, for any $I\in\Lat(A)$, the inclusion map $i:I\lhook\joinrel\longrightarrow A$ induces an order-embedding $\Cu_1(i):\Cu_1(I)\longrightarrow \Cu_1(A)$ and that $\Cu_1(I)$ is in fact a Scott-closed positively directed submonoid of $\Cu_1(A)$. The fact that $\Cu_1(I)$ is positively stable in $\Cu_1(A)$ is also trivial and left to the reader. We conclude that $\Cu_1(I)$ is an ideal of $\Cu_1(A)$ for any $I\in\Lat(A)$. 

We recall that for a $\CatCa$-algebra $A$, we let $\Lat(A)$ denote the complete lattice of ideals of $A$ and we let $\Lat_f (A)$ denote the sublattice of ideals in $A$ that contain a full, positive element. Also, for a positively directed and positively convex $\Cu^\sim$-semigroup $S$, we let $\Lat(S)$ denote the set of ideals of $S$ and we let $\Lat_f (S)$ denote the set of ideals in $S$ that are singly-generated by a positive element. (We might just write singly-generated, for notation purposes.)
\end{prg}

\begin{prop}
\label{thm:LatCu}
Let $A$ be $\CatCa$-algebra of stable rank one. Then the map 
\[
    \begin{array}{ll}
    		\hspace{-1,27cm}\Phi:\Lat(A)\longrightarrow \Lat(\Cu_1(A))\\
    		I\longmapsto \Cu_1(I)
    \end{array}
\]
is an isomorphism of complete lattices that maps $\Lat_f(A)$ onto $\Lat_f(\Cu_1(A))$. In particular, $A$ is simple if and only if $\Cu_1(A)$ is simple.
\end{prop}

\begin{proof}
Since $(x,k)\in \Cu_1(A)$ belongs to $\Cu_1(I)$ if and only if $x\in \Cu(I)$, the proof of \cite[Proposition 5.1.10]{APT14} remains valid in our context. For the sake of completeness, we explicitly write the inverse map \[
    \begin{array}{ll}
    		\hspace{-2,1cm}\Psi:\Lat(\Cu_1(A))\longrightarrow \Lat(A)\\
    		J\longmapsto \{x\in A \mid ([xx^*],0)\in J_+\}
    \end{array}
\]
where $J_+$ is the $\Cu$-semigroup formed by the positive elements of $J$.
\end{proof}

\begin{rmk}
(i) We explicitly compute the lattice structure on $\Cu_1(A)$ for any $\CatCa$-algebra $A$ of stable rank one. Let $I,J\in\Lat(A)$, then $\Cu_1(I)\wedge\Cu_1(J)= \Cu_1(I\cap J)$ and $\Cu_1(I)\vee\Cu_1(J)= \Cu_1(I+J)$. 

(ii) For a $\Cu^\sim$-semigroup $S$, we have that $\Lat(S)\simeq \Lat(S_+)$ and $\Lat_f(S)\simeq \Lat_f(S_+)$.

(iii) If $S$ is a countably-based $\Cu^\sim$-semigroup, then $\Lat_f(S)=\Lat(S)$.
\end{rmk}

\subsection{Link with \texorpdfstring{$\Cu$}{Cu} and \texorpdfstring{$\K_1$}{K1}} It has been shown in \cite{C20one} that the functor $\Cu$ and the functor $\K_1$ can be seen as the positive cone and the maximal elements of $\Cu_1$ respectively, through natural isomorphisms using the functors $\nu_+:\Cu^\sim\longrightarrow \Cu$ and $\nu_{max}:\Cu^\sim\longrightarrow \AbGp$. We now investigate further, applying these results at level of ideals and morphisms, in order to unravel the information contained within the functor $\Cu_1$, about the lattice of ideals of $\CatCa$-algebras of stable rank one and their morphisms.

\begin{lma}
\label{lma:commutsquareideal}
Let $S,T$ be positively directed and positively convex $\Cu^\sim$-semigroups. Let $\alpha:S\longrightarrow T$ be a $\Cu^\sim$-morphism and let $I,I'\in\Lat_f(S)$ be such that $I\subseteq I'$. Then:

$J:=I_{\alpha(e_{I_{max}})}$ and $J':=I_{\alpha(e_{I'_{max}})}$ are the smallest ideals of $T$ that contain $\alpha(I)$ and $\alpha(I')$respectively. Moreover, $J$ and $J'$ belong $\Lat_f(T)$ and $J\subseteq J'$. 
Thus, the following square is commutative:
\vspace{-0,2cm}\[
\xymatrix{
I\ar[d]_{\alpha_{|I}}\ar[r]^{i} & I'\ar[d]^{\alpha _{|I'}} \\
J\ar[r]_{i} & J'
}
\vspace{-0,2cm}\]
where $i$ stands for canonical inclusions and $\alpha_{|I}:I\longrightarrow J$ is the restriction of $\alpha$ that has codomain $J$, respectively $\alpha_{|I'}:I'\longrightarrow J'$.
\end{lma}

\begin{proof}
Since $\alpha$ is order-preserving, $\alpha_{|J}$ and $\alpha_{|J'}$ are well-defined. Besides, we know that for any $y\in I$, there exists $y'$, such that $0\leq y+y'\leq e_{I_{max}}$, hence  we have $0\leq \alpha(y)+\alpha(y')\leq \infty.\alpha(e_{I_{max}})$. Therefore $\alpha(y)\in J$ and we obtain that $\alpha(I)\subseteq J$, respectively $\alpha(I')\subseteq J'$. Since $I\subseteq I'$, we deduce that $e_{I_{max}}\leq e_{I'_{max}}$ and hence $\alpha(e_{I_{max}})\leq \alpha(e_{I'_{max}})$. Thus $J\subseteq J'$ which proves that the square is commutative.
\end{proof}

In the sequel, when we speak of the restriction of a $\Cu^\sim$-morphism to a singly-generated ideal, we will always refer, unless stated otherwise, to the map defined above. That is, we also restrict the codomain to the smallest singly-generated containing the image of the latter ideal.
Using notations of \autoref{lma:commutsquareideal}, notice that $\alpha_{|I}(e_{I_{max}})=e_{J_{max}}$.

\begin{prop}\cite[Proposition 5.5]{C20one}
\label{dfn:Sinfinite}
Let $\alpha:S\longrightarrow T$ be a $\Cu^\sim$-morphism between positively directed $\Cu^\sim$-semigroups $S,T$ that have maximal elements.
Then $\alpha_{max}:= \alpha_{|S_{max}} + e_{T_{max}}$ is a $\AbGp$-morphism from $S_{max}$ to $T_{max}$.
Thus we obtain a functor
\vspace{-0,23cm}\[
	\begin{array}{ll}
		\nu_{max}: \Cu^\sim \longrightarrow \AbGp\\
		\hspace{1,3cm} S \longmapsto S_{max}\\
		\hspace{1,35cm} \alpha \longmapsto \alpha_{max}
	\end{array}
\]
\end{prop}

In order to be well-defined as a functor, $\nu_{max}$ should have the full subcategory of positively directed $\Cu^\sim$-semigroups that have maximal elements as domain, that we also denote by $\Cu^\sim$. Observe that $\Cu_1(\CatCa_{\sr1,\sigma})$ belongs to the latter full subcategory, where $\CatCa_{\sr1,\sigma}$ is the full subcategory of separable $\CatCa$-algebras of stable rank one.

In the next theorem, we use the picture of the $\Cu_1$-semigroup described in \autoref{prg:newpicture}.

\begin{thm}\cite[Theorem 5.7]{C20one}
\label{thm:naturaltransfolma}
Let $A$ be either a separable or a simple $\sigma$-unital $\CatCa$-algebra of stable rank one. We have the following natural isomorphisms in $\Cu$ and $\AbGp$ respectively:
\vspace{-0cm}\[
	\hspace{0cm}\begin{array}{ll}
		\Cu_1(A)_+\simeq \Cu(A) \hspace{2,4cm} \Cu_1(A)_{max}\simeq \K_1(A)\\
		\hspace{0,35cm}(x,0)\longmapsto x \hspace{3,3cm} (\infty_A,k)\longmapsto k
	\end{array}
\]
In fact, we have the following natural isomorphisms: $\nu_+\circ\Cu_1\simeq \Cu$ and $\nu_{max}\circ\Cu_1\simeq \K_1$.\end{thm}

\begin{cor}
\label{cor:commutsquareidealA}
Let $A$ be either a separable or a simple $\sigma$-unital $\CatCa$-algebra of stable rank one. Let $I\in\Lat_f(A)$ be an ideal of $A$ that contains a full positive element and let $\phi:A\longrightarrow B$ be a $^*$-homomorphism. Write $\alpha:=\Cu_1(\phi)$ and $J:=\overline{B\phi(I)B}\in\Lat_f(B)$.  Let us use the notations of \autoref{prg:newpicture}, that is, $\alpha=(\alpha_0,\{\alpha_{I}\}_{I\in\Lat_f(A)})$. Then:

(i) $\nu_+(\alpha _{|\Cu_1(I)})={\alpha_0}_{|\Cu(I)}$ and $\nu_{max}(\alpha_{|I})=\alpha_I$.

(ii) Let $I'\in\Lat_f(A)$ such that $I'\supseteq I$. Then the following squares are commutative in their respective categories:
\vspace{-0,4cm}\[
\xymatrix{
\Cu(I)\ar[d]_{{\alpha_0}_{|\Cu(I)}}\ar[r]^{i} & \Cu(I')\ar[d]^{{\alpha_0}_{|\Cu(I')}} && \K_1(I)\ar[d]_{\alpha_{I}}\ar[r]^{\delta_{II'}} & \K_1(I')\ar[d]^{\alpha _{I'}}\\
\Cu(J)\ar[r]_{i} & \Cu(J')&& \K_1(J)\ar[r]_{\delta_{JJ'}} & \K_1(J') 
}
\]
where the maps $i$ stand for the natural inclusions in $\Cu$.
\end{cor}

\begin{proof}
(i) Using the isomorphisms of complete lattices of \autoref{thm:LatCu}, we get that $\Cu_1(J)$ belongs to $\Lat_f(\Cu_1(B))$ and is the smallest ideal of $\Cu_1(B)$ that contains $\alpha(\Cu_1(I))$. Hence, $\alpha_{|\Cu_1(I)}$ defined in \autoref{lma:commutsquareideal} has codomain $\Cu_1(J)$. We deduce that $\nu_+(\alpha_{|\Cu_1(I)})={\alpha_0}_{|\Cu_1(I)}$. Again, we write $\infty_J$ the maximal element of $\Cu(J)$. Finally, observe that $\nu_{max}(\alpha_{|I})(x,k)=(\alpha_0(x),\alpha_I(k))+(\infty_J,0)=(\infty_J,\alpha_I(k))$. 

(ii) Apply $\nu_+$ and $\nu_{max}$ to the square of \autoref{lma:commutsquareideal}, combined with the natural isomorphisms of \autoref{thm:naturaltransfolma} and condition (i) above to get the result.
\end{proof}

Observe that (ii) follows trivially from functoriality of $\Cu$ and $\K_1$ and also for any $I,I'\in\Lat(A)$ such that $I\subseteq I'$, but we illustrate here how it can also be derived from our methods. Furthermore, in order to be thorough, one would have to write $K_1(\phi_{|I}:I\longrightarrow J)$ instead of $\alpha_I$, since the latter map has only been defined for $I\in\Lat_f(A)$.

%%%%%%%%%%%Exactness around Cu1

\section{Quotients in the category \texorpdfstring{$\Cu^\sim$}{Cu-tilde} and exactness of the functor \texorpdfstring{$\Cu_1$}{Cu1}}
\label{sec:xactness}
\subsection{Quotients} We first study quotients of positively directed and positively convex $\Cu^\sim$-semigroups, to then show that the functor $\Cu_1$ preserves quotients. In other words, we prove that $\Cu_1(A)/\Cu_1(I)\simeq \Cu_1(A/I)$ for any $I\in\Lat(A)$.

\begin{dfn}
Let $S$ be a positively directed and positively convex $\Cu^\sim$-semigroup.
Let $I$ be an ideal of $S$. We define the following preorder on $S$: 
$x\leq_{I} y$ if there exists  $z \in I$ such that $x \leq z+y$. By antisymmetrizing this preorder, we get an equivalence relation on $S$, denoted $\sim_{I}$. We denote $\overline{x}:=[x]_{\sim_{I}}$.
\end{dfn} 

\begin{lma}
\label{lma:quotientideal}
Let $S$ be a positively directed and positively convex $\Cu^\sim$-semigroup. Let $I$ be an ideal of $S$. We canonically define
\hspace{0,2cm}$\left\{
\begin{array}{ll}
\overline{x}+ \overline{y}:=\overline{x+y}.\\
\overline{x}\leq \overline{y}$ \text{ if,} $x\leq_{I}y.
\end{array}
\right.$ 
and \hspace{0,2cm}$S/I:=(S/\!\!\sim_{I},+,\leq)$.

\vspace{0,2cm}Then $S/I$ is a positively directed and positively convex $\Cu^\sim$-semigroup. Also, $S \longrightarrow S/I$ is a surjective $\Cu^\sim$-morphism. 
\end{lma}

\begin{proof}
Let $x,y$ be in $S$. It is not hard to check that the sum and order considered are well-defined, that is, they do not depend on the representative chosen. Let us show that $S/I$ equipped with this sum and order belongs to $\CatoM$. Let $x_1,x_2$ and $y_1,y_2$ be elements in $S$ such that $\overline{x_1}\leq\overline{x_2} $ and $\overline{y_1}\leq \overline{y_2} $. There exist $z_1,z_2$ in $I$ such that $x_1+y_1\leq x_2+z_1+y_2+z_2$, that is, $\overline{x_1+y_1}\leq \overline{x_2+y_2} $. Also notice that the quotient map $S\longrightarrow S/I$ is naturally a surjective $\CatoM$-morphism.

In order to show that $(S/I,+,\leq)$ satisfies axioms (O1)-(04), and that $S\longrightarrow S/I$ is a $\Cu^\sim$-morphism, we proceed in a similar way as in \cite[Section 5.1]{APT14} for quotients in the category $\Cu$ and we will not get into too many details. This is based on the following two facts: 

(1) For any $\overline{x}\leq \overline{y}$ in $S/I$ there exist representatives $x,y$ in $S$ such that $x\leq y$. 

Indeed we know that there are representatives $x,y_1$ in $S$ and some $z\in I$ such that $x\leq y_1 +z$. Since $y:=(y_1+z)\sim_I y_1$, the claim is proved. 

(2) For any increasing sequence $(\overline{x_k})_k$ in $S/I$, we can find an increasing sequence of representatives $(x_k)_k$ in $S$. 

This uses (1) and the fact that $I$ satisfies (O1). Then $z:=\sup\limits_{n\in\N} (\sum\limits_{k=0}^n z_k)$, where $z_k$ are the elements obtained from (1), is an element of $I$. We refer the reader to \cite[\S 5.1.2]{APT14} for more details.

Let $\overline{x}\in S/I$ and let $x$ be a representative of $\overline{x}$ in $S$. We know there exists $p_x$ in $S$ such that $x+p_x\geq 0$. Since $0\in I$, we get that $\overline{x}+\overline{p_x}\geq \overline{0}$, that is, $S/I$ is positively directed.

Lastly, let $\overline{x}, \overline{y}\in S/I$ such that $\overline{x}\leq \overline{y}$ and $0\leq\overline{y}$.  Let $x$ be a representative of $\overline{x}$ and $y$ a representative of $\overline{y}$ in $S$. Then there are elements $z,w\in I$ such that $x\leq y+z$ and $0\leq y+w$. Since $I$ is positively directed, there exists $z'\in I$ such that $z+z'\geq 0$. Now observe that $x+w+z'\leq y+z+w+z'=(y+w)+(z+z')$ with $y+w+z+z'\geq 0$. By assumption $S$ is positively convex, hence we have $x+w+z'+y+w+z+z'\geq 0$ and thus in $S/I$ we obtain $\overline{x}+\overline{y}\geq 0$, as desired.
\end{proof}

A priori $(S/I,+,\leq)$ is not positively ordered either. Indeed, one could take for example an algebra that has a non-trivial ideal $I$ with no $\K_1$-obstructions and such that $\K_1(A)$ is not trivial. Then $\Cu_1(A)/\Cu_1(I)$ would not be positively ordered.

\begin{lma}
\label{prop:factorquotient}
Let $S,T$ be positively directed and positively convex $\Cu^\sim$-semigroups. Let $\alpha:S\longrightarrow T$ be a $\Cu^\sim$-morphism. For any $I\in\Lat(S)$ such that $I\subseteq \alpha^{-1}(\{0_T\})$, there exists a unique $\Cu^\sim$-morphism $\overline{\alpha}:S/I\longrightarrow T$ such that the following diagram is commutative:
\[
\xymatrix{
S\ar[dr]_{\pi}\ar[rr]^{\alpha} &&T\\
& S/I\ar@{-->}[ur]_{\overline{\alpha}} &
}
\]
As a matter of fact, we have $\overline{\alpha}(\overline{x}):=\alpha(x)$, where $x\in S$ is any representative of $\overline{x}$.
Moreover, $\overline{\alpha}$ is surjective if and only if $\alpha$ is surjective.
\end{lma}

\begin{proof}
By assumption $\alpha(I)=\{0\}$. Let us first prove that for any $x_1,x_2\in S$ such that $\overline{x_1}\leq\overline{x_2}$ in $S/I$, we have that $\alpha(x_1)\leq\alpha(x_2)$.
Let $x_1,x_2\in S$ be such that $x_1\lesssim_I x_2$. Then we know that there exists $z\in I$ such that $x_1\leq z+x_2$. Since $\alpha(z)=0$, we obtain that $\alpha(x_1)\leq\alpha(x_2)$.
We deduce that $\alpha$ is constant on the classes of $S/I$. Hence we can define $\overline{\alpha}:S/I\longrightarrow T$ by $\overline{\alpha}(\overline{x}):=\alpha(x)$, for any $x\in S$. By construction, the diagram is commutative. We only have to check that $\overline{\alpha}$ is a $\Cu^\sim$-morphism. Using facts (1) and (2) of the proof of \autoref{lma:quotientideal}, one can check that for any $\overline{x},\overline{y}\in S/I$ such that $\overline{x}\leq \overline{y}$ (resp $\ll$), there exists representatives $x,y$ in $S$ such that $x\leq y$ (resp $\ll$). Thus we easily obtain that $\overline{\alpha}$ is a $\Cu^\sim$-morphism which ends the first part of the proof. Surjectivity is clear and left to the reader.
\end{proof}

In the next theorem, we use the picture of the $\Cu_1$-semigroup described in \autoref{prg:newpicture}.

\begin{thm}
\label{thm:isoquotient}
Let $A$ be a $\CatCa$-algebra of stable rank one such that $\Lat_f(A)=\{\sigma\text{-unital ideals of }A\}$. \\Let $I\in\Lat(A)$. Let $\pi:A\longrightarrow A/I$ be the quotient map. Write $\pi^*:=\Cu_1(\pi):\Cu_1(A)\longrightarrow \Cu_1(A/I) $. Then $\pi^*((x,k))\leq \pi^*((y,l))$ if and only if $(x,k)\leq_{\Cu_1(I)}(y,l)$. Moreover $\pi^*$ is a surjective $\Cu^\sim$-morphism. Thus, this induces a $\Cu^\sim$-isomorphism 
\vspace{-0,1cm}\[
\Cu_1(A)/\Cu_1(I)\simeq \Cu_1(A/I).
\]
\end{thm}

\begin{proof}
Let us start with the surjectivity of $\pi^*$.
Let $[(a_I,u_I)]\in \Cu_1(A/I)$ where $a_I\in ((A/I)\otimes\mathcal{K})_+$ and $u_I$ is a unitary element of $(\her a_I)^\sim$. As $\pi$ is surjective, we know there exists $a \in A\otimes\mathcal{K}_+$ such that $\pi(a)=a_I$. Moreover, $\her a$ has stable rank one, hence unitary elements of $(\her (a_I))^\sim=\pi^\sim(\her (a)^\sim)$ lift. Thus, we can find a unitary element $u$ in $\her (a)^\sim$ such that $\pi^\sim(u)=u_I$. One can then check that $\pi^*([(a,u)])=[(a_I,u_I)]$. 

Let us show the first equivalence of the theorem. Noticing that $\pi^*(\Cu_1(I))=\{0_{\Cu_1(A/I)}\}$ and that $\pi^*$ is order-preserving, one easily gets the converse implication. 

Now let $(x,k)$ and $(y,l)$ be elements of $\Cu_1(A)$ such that $\pi^*((x,k))\leq\pi^*((y,l))$. We write $(\overline{x},\overline{k}):=\pi^*((x,k))=(\pi^*_0(x),\pi^*_x(k))$ and $(\overline{y},\overline{l}):=\pi^*((y,l))=(\pi^*_0(y),\pi^*_y(l))$. Thus we have $\overline{x}\leq\overline{y}$ in $\Cu(A/I)$. By \autoref{prg:quotientcu}, we know that $\Cu(A/I)\simeq\Cu(A)/\Cu(I)$, where the isomorphism is induced by the natural quotient map $\pi:A\longrightarrow A/I$. Therefore, there exists $z\in \Cu(I)$, such that $x\leq y+z$ in $\Cu(A)$. Write $y':=y+z$. Now by \autoref{cor:commutsquareidealA} and \cite[Proposition 4 (ii)]{LR95}, we obtain the following exact commutative diagram: 
\vspace{-0,5cm}\[
\xymatrix{
 & \K_1(I_{x}) \ar[d]_{\delta_{I_{x}I_{y'}}}\ar[r]^{\pi^*_{I_x}} & \K_1(I_{\overline{x}}) \ar[d]^{\delta_{I_{\overline{x}}I_{\overline{y}}}}\ar[r]^{} & 0
\\
 \K_1(I_z)\ar[r]_{\delta_{I_zI_{y'}}} & \K_1(I_{y'})\ar[r]_{\pi^*_{I_{y'}}} &  \K_1(I_{\overline{y}})\ar[r]^{} & 0
} 
\]
Thus, we get on the one hand that $\K_1(I_{y'})/  \delta_{I_zI_{y'}}(\K_1(I_z))\simeq \K_1(I_{\overline{y}})$ and on the other hand $\pi^*_{I_{y'}}\circ\delta_{I_{x}I_{y'}}= \delta_{I_{\overline{x}}I_{\overline{y}}}\circ \pi^*_{I_x} $. Moreover, by hypothesis, we have $\delta_{I_{\overline{x}}I_{\overline{y}}}(\overline{k})=\overline{l}$. So one finally gets that $\delta_{I_{x}I_{y'}}(k)=\delta_{I_{y}I_{y'}}(l)+\delta_{I_zI_{y'}}(l')$ for some $l'\in \K_1(I_z)$. That is, there exists $(z,l')\in \Cu_1(I)$ such that $(x,k)\leq (y,l) + (z,l')$. This ends the proof of the equivalence.

Finally, we already know that $\Cu_1(I)$ is an ideal of $\Cu_1(A)$ and that $\pi^*:\Cu_1(A)\twoheadrightarrow \Cu_1(A/I)$ is constant on classes of $\Cu_1(A)/\Cu_1(I)$. By \autoref{prop:factorquotient}, $\pi^*$ induces a surjective $\Cu^\sim$-morphism $\overline{\pi^*}:\Cu_1(A)/\Cu_1(I)\longrightarrow \Cu_1(A/I)$. Furthermore, the equivalence that we have just proved states that $\overline{\pi^*}$ is also an order-embedding. Thus we get a $\Cu^\sim$-isomorphism  $\Cu_1(A)/\Cu_1(I)\simeq \Cu_1(A/I)$. 
\end{proof}

\subsection{Exact sequences} We study the notion of exactness in the non-abelian category $\Cu^\sim$. From this, we show that $\Cu_1$ preserves short exact sequences of ideals, and we exhibit a split-exact sequence in $\Cu^\sim$ that links a positively ordered $\Cu^\sim$-semigroup that has maximal elements with its positive cone and its maximal elements. 

\begin{dfn}
\label{dfn:exactseq}
Let $S,T,V$ be positively directed $\Cu^\sim$-semigroups. Let $f:S\longrightarrow T$ be a $\Cu^\sim$-morphism. We define
$\im f:= \{ (t_1,t_2)\in T\times T: \exists s\in S, t_1\leq f(s)+t_2 \}$ and $\ker f:= \{ (s_1,s_2)\in S \times S: f(s_1)\leq f(s_2) \}$. 

Now consider $g:T\longrightarrow V$ a $\Cu^\sim$-morphism. We say that a sequence $... \longrightarrow S\overset{f}\longrightarrow T\overset{g}\longrightarrow V\longrightarrow ...$ is \emph{exact at $T$} if: $\ker g=\im f$. We say that it is \emph{short-exact} if $0 \longrightarrow S\overset{f}\longrightarrow T\overset{g}\longrightarrow V\longrightarrow 0 $ is exact everywhere. Finally, we say that a short-exact sequence is \emph{split}, if there exists a $\Cu^\sim$-morphism $q:V\longrightarrow S$ such that $g\circ q=id_V$.
\end{dfn}

\begin{prop}
\label{prop:blabla}
Let $S\overset{f}\longrightarrow T\overset{g}\longrightarrow V$ be a sequence in $\Cu^\sim$ as in \autoref{dfn:exactseq}. Then: 

(i) $f$ is an order-embedding if and only if   $0\longrightarrow S\overset{f}\longrightarrow T$ is exact. 

(ii) If $g$ is surjective then $T\overset{g}\longrightarrow V\longrightarrow 0$ is exact. If moreover $g(T)\in\Lat(V)$, then the converse is true.
\end{prop}

\begin{proof}
We recall that for $0\overset{0}\longrightarrow S$,  $\im0= \{(s_1,s_2)\in S^2 \mid s_1\leq s_2\}$ and that for $T\overset{0}\longrightarrow 0$, $\ker0= T^2$. Let us consider a sequence $S\overset{f}\longrightarrow T\overset{g}\longrightarrow V$ in $\Cu^\sim$.

(i) $f$ is an order-embedding if and only if $[s_1\leq s_2 \Leftrightarrow f(s_1)\leq f(s_2)]$, that is, if and only if $\im 0=\ker f$.

(ii) Suppose $g$ is surjective and let $v_1,v_2$ be elements in $V$. Since $V$ is positively directed we know that there exists an element $p_{v2}$ of $V$ such that $0\leq v_2+p_v$. Thus, we have $v_1\leq v_2+p_v+v_1$. By surjectivity, there exists $t\in T$ such that $g(t)=p_v+v_1$. Hence, for any $v_1,v_2$ in $V$ there exists $t\in T$ such that $v_1\leq g(t)+v_2$, that is, $\ker0=V^2=\im g$.

Suppose now that $T\overset{g}\longrightarrow V\longrightarrow 0$ is exact and that $g(T)$ is an ideal of $V$. We know that for any $v_1,v_2$, there exists $t\in T$ such that $v_1\leq g(t)+v_2$. In particular for $v_2=0$, we get that for any $v\in V$, there exists $t\in T$ such that $v\leq g(t)$. Moreover $g(T)$ is order-hereditary and thus $v\in g(T)$, which ends the proof. 
\end{proof}

\begin{lma}
\label{lma:bloblo}
Let $S\overset{f}\longrightarrow T\overset{g}\longrightarrow V$ be a sequence in $\Cu^\sim$. Assume that $f(S)$ is an ideal of $T$ such that $f(S)\subseteq g^{-1}(\{0_V\})$. By \autoref{prop:factorquotient}, we can consider $\overline{g}: T/f(S)\longrightarrow V $.
If $\overline{g}$ is a $\Cu^\sim$-isomorphism, then $S\overset{f}\longrightarrow T\overset{g}\longrightarrow V\longrightarrow 0$ is exact. If moreover $g(T)$ is an ideal of $V$, then the converse is true. 
\end{lma}

\begin{proof}
Suppose $T/f(S)\overset{\overline{g}}\simeq V$. Since $\overline{g}$ is an isomorphism, we know that  $g$ is surjective. Thus, by \autoref{prop:blabla}, we get exactness at $V$. Let us show exactness at $T$. We have the following equivalences: 

$(t_1,t_2)\in \ker g$ if and only if $g(t_1)\leq g(t_2)$ -by definition- if and only if $g(\overline{t_1})\leq g(\overline{t_2}) $ -since g is constant on classes of $T/f(S)$- if and only if $\overline{t_1}\leq \overline{t_2}$ -since $\overline{g}$ is an order-embedding- if and only if $t_1\leq f(s)+t_2$ for some $s\in S$ -by definition-, that is, if and only if $(t_1,t_2)\in \im f$.
\end{proof}

\begin{thm}
\label{thm:exactseqideal}
Let $A$ be a $\CatCa$-algebra of stable rank one such that $\Lat_f(A)=\{\sigma\text{-unital ideals of }A\}$. \\Let $I\in\Lat(A)$. Consider the canonical short exact sequence: $0 \longrightarrow I\overset{i}\longrightarrow A\overset{\pi}\longrightarrow A/I\longrightarrow 0 $. Then, the following sequence is short exact in $\Cu^\sim$:
\[
\xymatrix{
0\ar[r]^{} & \Cu_1(I)\ar[r]^{i^*} & \Cu_1(A)\ar[r]^{\pi^*} & \Cu_1(A/I)\ar[r]^{} & 0
} 
\]
\end{thm}

\begin{proof}
We know that $\Cu_1(I)$ is an ideal of $\Cu_1(A)$ and that $i^*$ is an order-embedding. Hence by \autoref{prop:blabla} (i), the sequence is exact at $\Cu_1(I)$. From \autoref{thm:isoquotient}, we also know that $\pi^*$ is constant on classes of $\Cu_1(A)/\Cu_1(I) $ and that $\overline{\pi^*}: \Cu_1(A)/\Cu_1(I)\simeq \Cu_1(A/I)$ is an isomorphism. Thus using \autoref{lma:bloblo} the result follows.
\end{proof}

\begin{cor}
\label{cor:exactunitiz}
Let $A$ be a $\CatCa$-algebra of stable rank one such that $\Lat_f(A)=\{\sigma\text{-unital ideals of }A\}$. Consider the canonical exact sequence $0 \longrightarrow A\overset{i}\longrightarrow A^\sim\overset{\pi}\longrightarrow A^\sim/A\simeq\mathbb{C}\longrightarrow 0$. Then there is a short exact sequence:
\vspace{-0,25cm}\[
\xymatrix{
0\ar[r]^{} & \Cu_1(A)\ar[r]^{i^*} & \Cu_1(A^\sim)\ar[r]^{\pi^*} & \overline{\N}\times\{0\}\ar[r]^{} & 0
} 
\]
\end{cor}

Now that we have numerous tools regarding ideals and exact sequences in $\Cu^\sim$, we will relate ideals, maximal elements, and positive cones through  exact sequences. Recall that for any positively directed $\Cu^\sim$-semigroup $S$ that has maximal elements, we have that $S_+\in \Cu$ and that $S_{max}\in\AbGp$; see \autoref{prop:PCABGP}.

Also, a $\Cu$-semigroup (respectively a $\Cu$-morphism) can be trivially seen as a $\Cu^\sim$-semigroup since $\Cu\subseteq \Cu^\sim$. The same can be done for an abelian group (respectively an $\AbGp$-morphism), -a fortiori, for the abelian group $S_{max}$ and the $\AbGp$-morphism $\alpha_{max}$- : Given $G\in\AbGp$, define $g_1\leq g_2$ if and only if $g_1=g_2$. From this, it follows that also $g_1\ll g_2$ if and only if $g_1=g_2$. This defines a functor $\AbGp\longrightarrow \Cu^\sim$ which allows us to see the category $\AbGp$ as a subcategory of $\Cu$.

Therefore, in what follows, we consider $\nu_+$ and $\nu_{max}$ as functors with codomain $\Cu^\sim$.
Finally, note that all of the proofs will be done in an abstract setting. Further, by \autoref{thm:naturaltransfolma}, we will be able to directly apply those results to $\Cu(A)$ and $\K_1(A)$, also seen as $\Cu^\sim$-semigroups.

\begin{dfn}
\label{dfn:vch}
Let $S$ be a positively directed $\Cu^\sim$-semigroup that has maximal elements. Let us define two $\Cu^\sim$-morphisms that link $S$ to $S_+$ on the one hand and to $S_{max}$ on the other hand, as follows:
\[
	\begin{array}{ll}
		i: S_+\overset{\subseteq}\hooklongrightarrow S \hspace{3cm} j:S \twoheadrightarrow S_{max}\\
		\hspace{0,7cm}s\longmapsto s \hspace{3,5cm}s\longmapsto s+e_{S_{max}}
	\end{array}
\]
\end{dfn}

\begin{thm}
\label{prop:vch}
Let $S$ be a positively directed $\Cu^\sim$-semigroup that has maximal elements. Consider the $\Cu^\sim$-morphisms defined in \autoref{dfn:vch}. Then $i$ is an order-embedding and $j$ is surjective. Moreover, the following sequence in $\Cu^\sim$ is split-exact: 
\vspace{-0cm}\[
\xymatrix{
0\ar[r]^{} & S_+\ar[r]^{i} & S\ar[r]^{j} & S_{max}\ar@{.>}@/_{-1pc}/[l]^{q}\ar[r]^{} & 0
} 
\vspace{-0,32cm}\]
where the split morphism is defined by $q(s):=s$.
\end{thm}

\begin{proof}
It is trivial to check that $i$ is a well-defined order-embedding $\Cu^\sim$-morphism.
We now need to check whether $j$ is a well-defined additive map. From \autoref{lma:PCequi}, we know that $s+e_{S_{max}}\in S_{max}$, for any $s\in S$. Also, because $2 e_{S_{max}}=e_{S_{max}}$, we get that $j$ is additive. Further, whenever $s\leq s'$, we know that  $s+e_{S_{max}}\leq s'+ e_{S_{max}}$. Since $s+e_{S_{max}} \in S_{max}$, we deduce that $j(s)=j(s')$ whenever $s\leq s'$. Further, $j(0)= e_{S_{max}}$. Thus, $j$ is a surjective $\Cu^\sim$-morphism.

By \autoref{prop:blabla}, we get exactness of the sequence at $S_+$ and $S_{max}$. Now let us check that the sequence is exact at $S$. Let $(s_1,s_2)\in \ker j$. Hence $j(s_1)=j(s_2)$, that is, $s_1+e_{S_{max}} = s_2+e_{S_{max}} $. Since $e_{S_{max}} \in S_+$, we easily get that $s_1\leq s_1+e_{S_{max}} = s_2+e_{S_{max}} $, which proves that $\ker j\subseteq \im i$. Conversely, let $(s_1,s_2)\in \im j$. Then we know that there exists a positive element $s\in S_+$ such that $s_1\leq s+s_2$. Since $e_{S_{max}}$ is the maximal positive element of $S$, we can take $s=e_{S_{max}}$. Then we easily get that $j(s_1)\leq j(s_2)$ -in fact, they are equal-. Thus we conclude that $\im i = \ker j$, which ends the proof.
\end{proof}

Note that we could not have used \autoref{lma:bloblo} here, since $S_+$ is not a $\Cu^\sim$ ideal of $S$. Indeed the smallest ideal containing $S_+$ is $S$ itself. We now give a functorial version of the latter split-exact sequence and also a likewise analogue for ideals.

\begin{cor}
\label{thm:chasingCuK1}
Let $S,T$ be positively directed $\Cu^\sim$-semigroups that have maximal elements. Let $\alpha:S\longrightarrow T$ be a $\Cu^\sim$ morphism. Viewing the functors $\nu_+$ and $\nu_{max}$ with codomain $\Cu^\sim$, we obtain the following commutative diagram with exact rows in $\Cu^\sim$:
\[
\xymatrix{
0\ar[r]^{} & S_+\ar[d]_{\alpha_+}\ar[r]^{i} & S \ar[d] _{\alpha}\ar[r]^{j} & S_{max} \ar[d]^{\alpha_{max}}\ar[r]^{} & 0
\\
0\ar[r]^{} & T_+\ar[r]_{i} & T\ar[r]_{j} & T_{max}\ar[r]^{} & 0
} 
\]
Furthermore, if $\alpha$ is a $\Cu^\sim$-isomorphism, then $\alpha_+$ is a $\Cu$-isomorphism and $\alpha_{max}$ is an abelian group isomorphism.
\end{cor}

\begin{proof}
We know from \autoref{prop:vch} that the row sequences are split-exact. Besides $\alpha_+=\alpha_{|S_+}$ hence the left square is commutative. Now take any $s\in S$. we have $\alpha_{max}\circ j_S(s) = \alpha_{max}(s+e_{S_{max}})= \alpha(s)+2e_{T_{max}}= \alpha(s)+e_{T_{max}}= j_T\circ\alpha(s)$, which proves that the right square is commutative.

Assume that $\alpha$ is an isomorphism. By functoriality, we obtain that $\alpha_+$ is a $\Cu$-isomorphism whose inverse is $(\alpha^{-1})_+$ and that $\alpha_{max}$ is an abelian group isomorphism whose inverse is $(\alpha^{-1})_{max}$.
\end{proof}

\begin{cor}
\label{cor:chasingCuK1ideal}
Let $S,T$ and $\alpha$ be as in \autoref{thm:chasingCuK1}. Assume also that $S,T$ are positively convex. Let $I$ be a singly-generated ideal of $S$ and $J:=I_{\alpha(e_{I_{max}})}$, the smallest (singly-generated) ideal of $T$ containing $\alpha(I)$ (see \autoref{lma:commutsquareideal}). We obtain the following commutative diagram with exact rows:
\[
\xymatrix{
0\ar[r]^{} & I_+\ar[d]_{({\alpha_{|I})}_+}\ar[r]^{i} & I \ar[d]_{{\alpha_{|I}}}\ar[r]^{j} & I_{max} \ar[d]^{({\alpha_{|I})}_{max}}\ar[r]^{} & 0
\\
0\ar[r]^{} & J_+\ar[r]_{i} & J\ar[r]_{j} & J_{max}\ar[r]^{} & 0
} 
\]
Furthermore, if $\alpha$ is a $\Cu^\sim$-isomorphism, then $\alpha(I)=J$ and $\alpha_{|I}:I\longrightarrow J$ is a $\Cu^\sim$-isomorphism. A fortiori, we also have $(\alpha_{|I})_+:I_+\longrightarrow J_+$ is a $\Cu$-isomorphism and $\alpha_I: I_{max}\longrightarrow J_{max}$ is an abelian group isomorphism.
\end{cor}
	
\begin{proof}
We only have to check that whenever $\alpha$ is an isomorphism, then $J=\alpha(I)$ and that $\alpha_{|I}:I\longrightarrow J$ defined as in  \autoref{lma:commutsquareideal} is an isomorphism. Then the conclusion will follow applying \autoref{thm:chasingCuK1} to $\alpha_{|I}$. Suppose that $\alpha$ is a $\Cu^\sim$-isomorphism. We know that $\alpha_{|I}:I\longrightarrow J$ sends any element $x\in I$ to $\alpha(x)\in J$. Since $\alpha$ is an order-embedding, so is $\alpha_{|I}$. By \autoref{lma:commutsquareideal}, we know that $\alpha(I)\subseteq J$ and that $\alpha(e_{I_{max}})=e_{J_{max}}$. Now since $\alpha$ is an isomorphism, we obtain that $\alpha^{-1}(e_{J_{max}})= e_{I_{max}} $. That is, by \autoref{lma:commutsquareideal}, $\alpha^{-1}(J)\subseteq I$. We deduce that $\alpha(I)=J$ and that $\alpha_{I}$ is a $\Cu^\sim$-isomorphism.
\end{proof}

We now transport the results obtained to concrete $\Cu^\sim$-semigroups, using \autoref{thm:naturaltransfolma}. 

\begin{thm}
\label{prg:transposetoalgebras}
Let $A,B$ be separable or simple $\sigma$-unital $\CatCa$-algebras of stable rank one. Let $\phi:A\longrightarrow B$ be a $^*$-homomorphism. Then the following diagram is commutative with exact rows:
\vspace{-0,21cm}\[
\xymatrix{
0\ar[r]^{} & \Cu(A)\ar[d]_{\Cu(\phi)}\ar[r]^{i} & \Cu_1(A) \ar[d] _{\Cu_1(\phi)}\ar[r]^{j} & \K_1(A) \ar[d]^{\K_1(\phi)}\ar[r]^{} & 0
\\
0\ar[r]^{} & \Cu(B)\ar[r]^{i} & \Cu_1(B)\ar[r]^{j} & \K_1(B)\ar[r]^{} & 0
} 
\vspace{-0,13cm}\]
Furthermore, if $\Cu_1(\phi)$ is a $\Cu^\sim$-isomorphism, then $\Cu(\phi)$ is a $\Cu$-isomorphism and $\K_1(\phi)$ is a $\AbGp$-isomorphism.\\

\vspace{-0,4cm}Let $I\in\Lat(A)$. Write $J:=\overline{B\phi(I)B}$, the smallest ideal of $B$ containing $\phi(I)$ and $\alpha:=\Cu_1(\phi)$. We denote $\alpha=(\alpha_0,\{\alpha_{I}\}_{I\in\Lat(A)})$ as in \autoref{prg:newpicture}. Then the following diagram is commutative with exact rows:
\vspace{-0,21cm}\[
\xymatrix{
0\ar[r]^{} & \Cu(I)\ar[d]_{{\alpha_0}_{|\Cu(I)}}\ar[r]^{i} & \Cu_1(I) \ar[d] _{\alpha_{|\Cu_1(I)}}\ar[r]^{j} & \K_1(I) \ar[d]^{\alpha_I}\ar[r]^{} & 0
\\
0\ar[r]^{} & \Cu(J)\ar[r]^{i} & \Cu_1(J)\ar[r]^{j} & \K_1(J)\ar[r]^{} & 0
} 
\vspace{-0,12cm}\]
Furthermore, if $\alpha$ is a $\Cu^\sim$-isomorphism, then $\alpha(\Cu_1(I))=\Cu_1(J)$ and $\alpha_{|\Cu_1(I)}:\Cu_1(I)\longrightarrow \Cu_1(J)$ is a $\Cu^\sim$-isomorphism. A fortiori, we also have ${\alpha_0}_{|\Cu(I)}:\Cu(I)\longrightarrow \Cu(J)$ is a $\Cu$-isomorphism and $\alpha_I: \K_1(I)\longrightarrow \K_1(J)$ is an $\AbGp$-isomorphism.
\end{thm}

\vspace{-0,3cm}\begin{proof}
Combine \autoref{thm:chasingCuK1} and \autoref{cor:chasingCuK1ideal} with \autoref{lma:commutsquareideal}.
\end{proof}

\vspace{-0,15cm}
\end{document}